\documentclass[a4paper, 11pt]{article}
\usepackage{amsmath}
\usepackage{amssymb,esint}
\usepackage{amscd}
\usepackage{xspace}
\usepackage{fancyhdr}
\usepackage{color}
\usepackage{srcltx} 

\newtheorem{theorem}{Theorem}[section]

\newtheorem{definition}[theorem]{Definition}

\newtheorem{lemma}[theorem]{Lemma}

\newtheorem{remark}[theorem]{Remark}

\newenvironment{proof}[1][Proof]{\textbf{#1.} }{\hfill\rule{0.5em}{0.5em}}
{\catcode`\@=11\global\let\AddToReset=\@addtoreset
\AddToReset{equation}{section}

\AddToReset{theorem}{section}

\newcommand{\cM}{{\mathcal M}}
\newcommand{\cP}{{\mathcal P}}

\newcommand{\R}{{\mathbb R}}

\newcommand{\N}{{\mathbb N}}

\newcommand{\bW}{{\mathbf W}}  
\newcommand{\bm}{{\mathbf m}}

\newcommand{\argmin}{{\rm argmin}}

\usepackage{enumerate}

\begin{document}
\title{Modified Wasserstein gradient flow formulation of time-fractional porous medium equations with nonlocal pressure}
	\author{
	{Nhan-Phu Chung\thanks{E-mail address: phuchung82@gmail.com, Institute of Applied Mathematics, University of Economics Ho Chi Minh City, Vietnam.} , Thanh-Son Trinh\thanks{E-mail address: trinhthanhson@iuh.edu.vn, Faculty of Information Technology, Industrial University of Ho Chi Minh City, Vietnam.} } }
\date{}  
\maketitle
\begin{abstract}
We consider a class of time-fractional porous medium equations with nonlocal pressure. We show the existence of their weak solutions by proposing a JKO scheme for modified Wasserstein distance and a square fractional Sobolev norm. Moreover, the regularization effect and the $L^p$ norm estimate are established in this paper.

\noindent
\end{abstract}  
\section{Introduction}
In this paper, we investigate the following two problems
\begin{align}\label{E-time-fractional eq}
\partial_t^\alpha u=\textup{div}(u^{\beta}\nabla(-\Delta)^{-s}u) \mbox{ in } \R^d,
\end{align} 
\begin{align}\label{E-beta=0}
\partial^\alpha_t u=-(-\Delta)^{1-s}u \text{ in }\mathbb{R}^d,
\end{align} 
where $0<\alpha<1, 0<\beta\leq 1, d\geq 1$ and $0<s<\min\{1,d/2\}$. Note that when $\beta\to 0$ then the equation \eqref{E-time-fractional eq} degenerates to the equation \eqref{E-beta=0}. 

 The fractional-order derivative $\partial^\alpha_t u$ is defined as 
$$\partial^\alpha_t u(t):=\frac{1}{\Gamma(1-\alpha)}\int_0^t(t-r)^{-\alpha}u'(r)dr,$$ 
with $\Gamma$ is the Gamma function defined by $\Gamma(z):=\int_0^\infty r^{z-1}e^{-r}dr.$ For $0<r<\min\{1,d/2\}$, the $r$-fractional Laplacian $(-\Delta)^r$ on $\R^d$ is defined by means of Fourier transform as \begin{align*}
(\widehat{(-\Delta)^r u})(\xi)=| \xi |^{2r}\hat{u}(\xi).
\end{align*}

Problem \eqref{E-time-fractional eq}, with $\alpha=1$ and $\beta=1$, has been studied by Caffarelli and V\'{a}zquez in \cite{CV}. In particular, when $s=0$, we get the standard porous medium equation (see more details in \cite{A, Vq}). This model arises from considering the continuity equation of a nonnegative density distribution $u(x,t)$ according to Darcy’s law
\begin{align*}
\partial_tu+\text{div}(u\mathbf{v})=0,
\end{align*}
with $\mathbf{v}=\nabla \mathbf{p}$ is the potential, and $\mathbf{p}$ is the pressure. There are different relations between the density distribution $u$ and the pressure $\mathbf{p}$. For example, the model was proposed by Leibenzon and Muskat in 1930s, takes the relation in the form $\mathbf{p}=f(u)$, with $f$ is a nondecreasing scalar function. In this paper, we consider the case of nonlocal pressure that $\mathbf{p}=(-\Delta)^{-s}u$ with $0<s<\min\{1,d/2\}$. The equation \eqref{E-time-fractional eq} with $\alpha=1$ and the nonlocal pressure has been studied for the case $0<\beta<2$ \cite{STV-JDE} and for all $\beta >0$ \cite{NV, STV-ARMA}. In particular, in \cite{STV-JDE}, authors show that when $\beta\in [0,1)$ the problem \eqref{E-time-fractional eq} with $\alpha=1$, has infinite speed of propagation, and for $\beta\in [1,+\infty)$ it has finite speed of propagation. Problem \eqref{E-beta=0}, with $\alpha=1$, has been studied by Erbar in \cite{Erbar}, and Chung and Nguyen in \cite{CN} by two different methods.

In 1999 \cite{Caputo}, by modifying the Darcy's law, Caputo introduced and investigated the following equation 
\begin{align*}
\partial_t^\alpha u-\text{div} (\kappa (u)\nabla u)=f,
\end{align*}
with $\partial_t^\alpha u$ denotes the Caputo fractional derivative of order $\alpha$. It has many applications in other fields such as physics, engineering, viscoelasticity, porous media, ... \cite {KRY, LLW, MP}. Recent years, time-fractional diffusion equations has been studied and developed by many authors \cite{DN,DJ,KSVZ,KSZ}.

In \cite{ACV}, Allen, Caffarelli and Vasseur studied the following equation \begin{align}\label{E-time-fractional with f}
\partial_t^\alpha u-\textup{div}(u^{\beta}\nabla(-\Delta)^{-s}u) = f \mbox{ in } \R^d.
\end{align} 
They proved the existence of weak solutions for the equation \eqref{E-time-fractional with f} for the case $\beta=1$, $0<s<\frac{1}{2}$ when $f$ and the initial data have exponential decay at infinity. Later, Djida, Nieto and Area \cite{DNA} extended results in \cite{ACV} for the case $\beta\geq 1, 0<s<\frac{1}{2}$. In a recent paper \cite{DN}, Dao and Nguyen handled the equation \eqref{E-time-fractional with f} for the case $\beta\geq 1$, $\frac{1}{2}\leq s<1$ and $f=0$. In this article, we study the equation \eqref{E-time-fractional with f} for the case $0\leq \beta\leq 1$, $0<s< 1$ and $f=0$. Our approaches in solving \eqref{E-time-fractional with f} for this case are different from \cite{ACV, DN, DNA}.

To solve problems \eqref{E-time-fractional eq} and \eqref{E-beta=0}, we propose new JKO schemes on modified Wasserstein distances. After the works of Jordan, Kinderlehrer and Otto on their seminal paper \cite{JKO}, their JKO schemes via the usual Wassertein distance in the space of probability measures have become a powerful tool to study a numerous classes of PDEs \cite{AmGiSa, AMS, AS, BB, CV, CLSS, DS, Erbar, LMS, LMS12, Otto, OttoWest}. Specially, in the recent paper in 2020, Duong and Jin \cite{DJ} are the first ones applying a JKO scheme on the usual Wasserstein distance to solve a class of time-fractional PDE. On the other hand, the modified Wasserstein distance was introduced and investigated in \cite{DNS, CLSS}. It has been applied to studied in several classes of PDEs \cite{CN, LMS12, MMS}.

A function $u:[0,\infty)\to \cP(\R^d)$ is a weak solution of the equation
\begin{align*}
\partial_t^\alpha u=\textup{div}\left(u^{\beta}\nabla(-\Delta)^{-s}u\right)
\end{align*}  
if for every $T>0$ and every $\phi\in C^\infty \left([0,T]\times\R^d\right)$ satisfying $\phi(T)=0$ and $\phi(t,\cdot)\in C^\infty_c(\R^d)$ for every $t\in [0,T]$, we have $$\int_0^T\int_{\R^d}\left(\partial^\alpha_t\phi(t)u(t)+\nabla_x \phi(t)u^{\beta}\nabla(-\Delta)^{-s}u(t)\right) dxdt=\frac{1}{\Gamma(1-\alpha)}\int_{\R^d}u(0)\int_0^Tt^{-\alpha}\phi(t)dtdx.$$
Similarly, a function $u:[0,\infty)\to \cP(\R^d)$ is a weak solution of the equation \begin{align*}
\partial_t^\alpha u=-(-\Delta)^{1-s}u
\end{align*}  
if for every $T>0$ and $\phi\in C^\infty \left([0,T]\times\R^d\right)$ satisfying $\phi(T)=0$ and $\phi(t,\cdot)\in C^\infty_c(\R^d)$ for every $t\in [0,T]$, we have $$\int_0^T\int_{\R^d}\left(\partial^\alpha_t\phi(t)+(-\Delta)^{1-s}\phi (t)\right) u(t)dxdt=\frac{1}{\Gamma(1-\alpha)}\int_{\R^d}u(0)\int_0^Tt^{-\alpha}\phi(t)dtdx.$$

Now let us present our new JKO schemes. 
Let $\bm:[0,+\infty)\to (0,+\infty)$ be a map such that $\inf_{x\in [0,+\infty)} \bm(x)>0$ and $\bm''\leq 0$. Then we can define the modified Wasserstein distance 
$\bW^2_\bm$ on the space $\cP_2(\R^d)$ consisting of all Borel probability measures on $\R^d$ with finite second moments (see Section 2.1). 
Given $\tau>0$ and $u^0_\tau:=u_0\in \cP_2(\R^d)$, we define $u^k_\tau$ inductively for $k\in \N$ as follows
\begin{align}\label{JKO schemes}
u^{k}_\tau:=\argmin_{u\in \cP(\R^d)}\bigg\{\frac{C_\alpha}{2\tau^\alpha}\bW^2_\bm(u,\overline{u}^{k-1}_\tau)+\frac{1}{2}\|u\|^2_{\dot{H}^{-s}(\R^d)}\bigg\},
\end{align} 
where $C_\alpha:=\frac{1}{\Gamma(2-\alpha)}$, $\|\cdot\|_{\dot{H}^{-s}(\R^d)}$ is the norm of the homogeneous Sobolev space $\dot{H}^{-s}(\R^d)$, and $\overline{u}^{k-1}_\tau:=\sum_{i=0}^{k-1}\left(-b^{(k)}_{k-i}\right)u^i_\tau$ with $b^{(k)}_{i}$ is defined by \begin{align*}
b^{(k)}_{i}:=\left\{\begin{array}{ll}
1, & i=0,\\
(i+1)^{1-\alpha}+(i-1)^{1-\alpha}-2i^{1-\alpha},&i=1,\ldots,k-1,\\
(k-1)^{1-\alpha}-k^{1-\alpha}, & i=k.
\end{array}\right.
\end{align*} 

Then we apply our JKO schemes \eqref{JKO schemes} for $\bm(z)=\left(z+\tau^{\frac{\alpha}{4(2-\beta)}}\right)^{\beta}$ and $\bm(z)=(z+1)^{\tau^{1-\alpha/4}}$ to solve equation \eqref{E-time-fractional eq} and equation \eqref{E-beta=0}, respectively. The two main technical challenges of our JKO scheme is to deal with the nonlocality of the fractional derivative $\partial_t^\alpha u $ and the degeneracy of both equations \eqref{E-time-fractional eq} and \eqref{E-beta=0}. To overcome the first one we adopt the piecewise linear approximation, known as the L1 approximation \cite{LX}, which was used before in \cite{DJ} to solve the time-fractional Fokker-Planck equation. As our equations are degenerated we can not use the usual Wasserstein distance in our scheme as \cite{DJ}. To overcome this issue, we employ our weight functions $\bm$ (thus the modified Wasserstein distance) depending on the time-step $\tau$. Incorporating the time-step into a transport cost functional allows us adapt known techniques \cite{JKO, LMS12, MMS, OttoWest} for the convergence analysis. Note that this idea was used before for solving several classes of PDEs \cite{CN, DPZ, H,LMS,PRST}. To our best knowledge, no one have used JKO schemes with modified Wasserstein distances to solve time-fractional PDEs before. Our main contributions of this article are gradient formulation of the scheme \eqref{JKO schemes} and its convergence analysis.

We now illustrate our main results in this paper. Our first main result is the following Theorem \ref{main result}, which we prove the existence of weak solutions of equation \eqref{E-time-fractional eq}. Furthermore, we also show the regularity estimate of our interpolation approximation.
\begin{theorem}\label{main result}
Let $u_0\in \cP_2(\R^d)\cap \dot{H}^{-s}(\R^d)$, $d\geq 2$, $0<\beta\leq 1$, $0<s<\min\{1,\frac{d}{2}\}$. For every $\tau>0$, let $\{{u^k_\tau}\}_{k\in \N}$ be the solution of scheme \eqref{JKO schemes} with $\bm(z)=\left(z+\tau^{\frac{\alpha}{4(2-\beta)}}\right)^{\beta}$. We define the interpolation function $\hat{u}_\tau:[0,\infty)\to \cP(\R^d)$ by \begin{align}
\label{appro. of u} \hat{u}_\tau(t):=u^k_\tau \mbox{ for every } (k-1)\tau<t\leq k\tau. \end{align}
Then there exists a function $u:[0,\infty)\to \cP(\R^d)$ such that 
\begin{enumerate}
\item for every $T>0$, $\hat{u}_{\tau} $ converges weakly to $u$ in $L^2((0,T);\dot{H}^{1-s}(\R^d))$ as $\tau\to 0$ ; 
\item and $u$ is a weak solution of the equation
\begin{align*}
\partial_t^\alpha u=\textup{div}\left(u^{\beta}\nabla(-\Delta)^{-s}u\right).
\end{align*}  
\end{enumerate}

3. On the other hand, there exists $K_1>0$ such that for every $t\geq 0$ and $1\leq q<p$, we have \begin{align}\label{F-regularity}
\| \hat{u}_\tau(t)\|_{L^p(\R^d)}\leq \min\bigg\{\|u_0\|_{L^p(\R^d)}, K_1\bigg(\frac{(\beta+p)^2}{4p(p-1)}\tau^{-\alpha}\|u_0\|^{\eta_2}_{L^q(\R^d)}\|u_0\|_{L^p(\R^d)}^{p}\bigg)^{\frac{1}{\eta_1}} \bigg\},
\end{align}
where \begin{align*}
\eta_1: = \frac{\frac{\beta +p}{q}-1+\frac{2(1-s)}{d}}{\frac{1}{q}-\frac{1}{p}}, \;\eta_2 := \frac{\frac{\beta }{p}+\frac{2(1-s)}{d}}{\frac{1}{q}-\frac{1}{p}}.
\end{align*}

\end{theorem}

In our second main result, we get similar results for equation \eqref{E-beta=0}.
\begin{theorem}\label{main result2}
Let $u_0\in \cP_2(\R^d)\cap \dot{H}^{-s}$, $0<\beta\leq 1$, $0<s<\min\{1,\frac{d}{2}\}$. For every $\tau>0$, let $\{{u^k_\tau}\}_{k\in \N}$ be the solution of scheme \eqref{JKO schemes} with $\bm(z)=(z+1)^{\tau^{1-\alpha/4}}$. 
Then there exists a function $u:[0,\infty)\to \cP(\R^d)$ such that 
\begin{enumerate}
\item for every $T>0$, $\hat{u}_{\tau} $ converges weakly to $u$ in $L^2((0,T);\dot{H}^{1-s}(\R^d))$ as $\tau\to 0$ ; 
\item and $u$ is a weak solution of the equation
\begin{align*}
\partial_t^\alpha u=-(-\Delta)^{1-s}u.
\end{align*}  
\end{enumerate}
3. On the other hand, there exists $K_2>0$ such that for every $\tau \leq 1$, $t\geq 0$ and $1\leq q<p$, we have 
\begin{align}\label{F-regularity-beta=0}
\| \hat{u}_\tau(t)\|_{L^p(\R^d)}\leq \min\bigg\{\|u_0\|_{L^p(\R^d)}, K_2\bigg(\frac{(p+1)^2}{4p(p-1)}\tau^{-\alpha}\|u_0\|^{\theta_2}_{L^q(\R^d)}\|u_0\|_{L^p(\R^d)}^{p}\bigg)^{\frac{1}{\theta_1}} \bigg\},
\end{align}
where 
 \begin{align*}
\theta_1: = \frac{\frac{\tau^{1-\alpha/4} +p}{q}-1+\frac{2(1-s)}{d}}{\frac{1}{q}-\frac{1}{p}}, \;\theta_2 := \frac{\frac{\tau^{1-\alpha/4} }{p}+\frac{2(1-s)}{d}}{\frac{1}{q}-\frac{1}{p}}.
\end{align*}

\end{theorem}

Our paper is organized as follows. In section 2, we review  the fractional calculus, fractional Sobolev spaces and the definition of modified Wasserstein distance. In the next section, we show the existence and uniqueness for solutions of our JKO scheme, and we also establish the framework for gradient flows in the modified Wasserstein space. We give the proof of our main results Theorems \ref{main result} and \ref{main result2} in section 4. 
\section{Preliminaries}
\subsection{Modified Wasserstein distance}
In this subsection, we review the modified Wasserstein distance which was introduced and investigated in \cite{DNS, CLSS}.

First, we recall the definition of the Wasserstein distance \cite{V}. In this paper, we consider probability measures on $\R^d$ that are absolutely continuous with respect to Lebesgue measure and identify a probability measure with its density. We denote by $\cP(\R^d)$ the set of all probability Borel measures on $\R^d$. The set $\cP_2(\R^d)$ is defined as the set of all measure $\mu\in \cP(\R^d)$ with finite second moment, i.e., \begin{align*}
\cP_2(\R^d):=\bigg\{\mu: \R^d\to [0,+\infty)\text{ is measurable}: \int_{\R^d}\mu(x)dx=1\text{ and } \int_{\R^d}\vert x\vert^2d\mu(x)<\infty\bigg\}.
\end{align*}
For every $\gamma_1,\gamma_2\in \cP_2(\R^d)$, we denote by $\Pi(\gamma_1,\gamma_2)$ the set of all probability Borel measure $\pi$ on $\R^d\times \R^d$ such that the first and the second marginal of $\pi$ are $\gamma_1$ and $\gamma_2$, respectively, i.e., for every Borel subset $A$ of $\R^d$, we have\begin{align*}
\pi(A\times \R^d)=\gamma_1(A) \text{ and }\pi(\R^d\times A)=\gamma_2(A).
\end{align*}
Then, the Wasserstein distance between $\nu_1$ and $\nu_2$ is defined by \begin{align*}
W_2(\gamma_1,\gamma_2):=\bigg(\inf_{\pi\in \Pi(\gamma_1,\gamma_2)}\int_{\R^d\times \R^d}\vert x-y\vert^2d\pi(x,y)\bigg)^{1/2}.
\end{align*}
In \cite{BB}, by considering the continuity equation $\partial_t\rho_t+\text{div}(\rho_tv_t)=0$, Benamou and Brenier give the equivalence definition for Wasserstein distance as follows.
\begin{align*}
W_2(\gamma_1,\gamma_2 )=\inf\bigg\{\int_0^t\int_{\R^d}\rho_t(x)|v_t(x)|^2dxdt: \partial_t\rho_t+\text{div}(\rho_tv_t)=0, \gamma_1=\rho_0\mathcal{L}^d, \gamma_2=\rho_1\mathcal{L}^d\bigg\}^{1/2},
\end{align*}
where $\mathcal{L}^d$ is the Lebesgue measure of $\R^d$.

Now, we present the definition of the modified Wasserstein distance. We consider the following continuity equation \begin{align}\label{CE}
\partial_t\mu_t+\nabla\cdot \nu_t=0 \text{ in }\R^d\times (0,1),
\end{align}
where families measures $(\mu_t)_{t\in [0,1]}$ in the space $\cM^+_{\text{loc}}(\R^d)$ of nonnegative Radon measures in $\R^d$, and $(\nu_t)_{t\in [0,1]}$ in the space $\cM_{\text{loc}}(\R^d, \R^d)$ of $\R^d$-valued Radon measures in $\R^d$. Then, we denote by $\mathcal{CE}$ the set of $((\mu_t)_{t\in [0,1]},(\nu_s)_{s\in [0,1]})\subset \cM^+_{\text{loc}}(\R^d)\times \cM_{\text{loc}}(\R^d, \R^d)$ satisfying the following three conditions.
\begin{enumerate}[(i)]
	\item $t\mapsto \mu_t$ is weakly* continuous in $\cM^+_{\text{loc}}(\R^d)$,
	\item $(\nu_t)_{t\in [0,1]}$ is a Borel family and $\int_0^1\vert \nu_t\vert (\R^d)dt<\infty$,
	\item $(\mu_t,\nu_t)_{t\in [0,1]}$ is a solution of \eqref{CE} in the sense of distributions, i.e., \begin{align*}
\int_0^1\int_{\R^d}\partial_t\varphi (x,t)d\mu_t(x)dt+\int_0^1\int_{\R^d}\nabla_x\varphi (x,t)\cdot d\nu_t(x)dt=0,\text{ for every }\varphi\in C_c^1(\R^d\times (0,1)).
\end{align*}
\end{enumerate}
For measures $\gamma^0,\gamma^1\in \cM^+_{\text{loc}}(\R^d)$, $\mathcal{CE}(\gamma^0\to \gamma^1)$ denotes the subset of $\mathcal{CE}$ such that $\mu_0=\gamma^0$ and $\mu_1=\gamma^1$.\\
Let $\mathbf{m}: [0,\infty)\to (0,\infty)$ be a concave and nondecreasing function, we define the action density function $\zeta:[0,\infty)\times \R^d\to [0,\infty)$ by \begin{align}\label{density func}
\zeta(r,s):=\frac{\vert s\vert^2}{\bm(r)}, \text{ for every } (r,s)\in [0,\infty)\times \R^d.
\end{align}
Then for every $\gamma^0,\gamma^1\in \cM^+_{\text{loc}}(\R^d)$, we define 
\begin{align*}
\bW_\bm(\gamma^0,\gamma^1):=\inf_{(\mu,\nu)\in\mathcal{CE}(\gamma^0\to \gamma^1)}\bigg(\int_0^1\mathcal{S} (\mu_t,\nu_t)dt\bigg)^{1/2},
\end{align*}  
if $\mathcal{CE}(\gamma^0\to \gamma^1)$ is nonempty, and $\bW_\bm(\gamma^0,\gamma^1)=+\infty$, otherwise. Here, the function $\mathcal{S}(\mu,\nu)$ is defined by \begin{align*}
\mathcal{S}(\mu,\nu):=\begin{cases}\int_{\R^d}\zeta(\sigma,\varrho)dx &\text{ if }\nu^\perp=0,\\ +\infty &\text{ otherwise,} \end{cases}
\end{align*}
where $\mu=\sigma\mathcal{L}^d+\mu^\perp$ and $\nu=\varrho\mathcal{L}^d+\nu^\perp$ are the Lebesgue decompositions with $\mathcal{L}^d$ is the Lebesgue measure of $\R^d$.\\
If $\mathcal{CE}(\gamma^0\to \gamma^1)$ is nonempty, applying \cite[Theorem 5.4]{DNS}, we also have that
\begin{align*}
\bW_\bm(\gamma^0,\gamma^1):=\inf_{(\mu,\nu)\in\mathcal{CE}(\gamma^0\to \gamma^1)}\int_0^1\mathcal{S} (\mu_t,\nu_t)^{1/2}dt.
\end{align*}  
From our assumptions on the function $\bm$, applying \cite[Theorem 3.1 and 3.3]{DNS} we have that our action density function $\zeta$ satisfies conditions (3.1a), (3.1b) and (3.1c) in \cite{DNS}. Therefore $\bW_\bm$ is a modified Wasserstein (pseudo) distance on $\cM^+_{\text{loc}}(\R^d)$ in the sense of \cite[Definition 5.1]{DNS}.

For our JKO scheme \eqref{JKO schemes}, given $\tau>0$, we use weight functions $\bm(z)=\left(z+\tau^{\frac{\alpha}{4(2-\beta)}}\right)^{\beta}$ and $\bm(z)=(z+1)^{\tau^{1-\alpha/4}}$ to solve equations \eqref{E-time-fractional eq} and \eqref{E-beta=0}, respectively. As $0<\alpha<1, 0<\beta\leq 1$ then for every $0<\tau<1$, our weight functions $\bm(z)=\left(z+\tau^{\frac{\alpha}{4(2-\beta)}}\right)^{\beta}$ and $\bm(z)=(z+1)^{\tau^{1-\alpha/4}}$ are concave and nondecreasing.
\subsection{Fractional Sobolev spaces}
First, we recall the definition of the Caputo derivative (for more details, readers can see \cite{KST}). Let $\alpha\in (0,1)$ and $t_l,t_r\in\R$ such that $t_l<t_r$. We define the left-sided and right-sided Caputo derivative of order $\alpha$ of a function $\varphi: (t_l,t_r)\to\R$, respectively by 
\begin{align*}
{}_{t_l}D^{\alpha}_t\varphi(t):=\frac{1}{\Gamma(1-\alpha)}\int_{t_l}^t(t-s)^{-\alpha}\varphi '(s)ds\text{ and }{}_{t}D^{\alpha}_{t_r}\varphi(t):=\frac{1}{\Gamma(1-\alpha)}\int_{t}^{t_r}(s-t)^{-\alpha}\varphi '(s)ds.
\end{align*}
For simplicity, we write $\partial^\alpha_t \varphi$ instead of ${}_0D^\alpha_t\varphi$.

Let us now review the fractional Sobolev spaces \cite{BCD}. For $\varphi \in L^1(\R^d)$, the Fourier transform of $\varphi$ is defined as \begin{align*}
\hat{\varphi}(\xi):=\int_{\R^d}e^{-ix\cdot \xi}\varphi(x)dx.
\end{align*}
$\mathcal{S}(\R^d)$ denotes the Schwartz space of smooth function on $\R^d$ with rapid decay at infinity and $\mathcal{S}'(\R^d)$ denotes its dual space. Let $r\in \R$, for every $\varphi\in \mathcal{S}'(\R^d)$ such that $\varphi \in L^1_{\text{loc}}(\R^d)$, we define
\begin{align*}
&\|\varphi\|^2_{H^r(\R^d)}:=\frac{1}{(2\pi)^d}\int_{\R^d}(1+\vert \xi\vert^2)^r\vert \hat{\varphi}(\xi)\vert^2d\xi,\\
&\|\varphi\|^2_{\dot{H}^r(\R^d)}:=\frac{1}{(2\pi)^d}\int_{\R^d}\vert \xi\vert^{2r}\vert \hat{\varphi}(\xi)\vert^2d\xi .
\end{align*}
Then, we define the fractional Sobolev space $H^r(\R^d)$ and the homogeneous fractional Sobolev space $\dot{H}^r(\R^d)$ respectively by 
\begin{align*}
H^r(\R^d)&:=\{\varphi\in \mathcal{S}'(\R^d):  \varphi \in L^1_{\text{loc}}(\R^d), \|\varphi\|^2_{H^r(\R^d)}<+\infty\},\\
\dot{H}^r(\R^d)&:=\{\varphi\in \mathcal{S}'(\R^d):  \varphi \in L^1_{\text{loc}}(\R^d), \|\varphi\|^2_{\dot{H}^r(\R^d)}<+\infty\}.
\end{align*}
If $r<d/2$ then for every $u,v\in \dot{H}^r(\R^d)$, the scalar product $\langle u,v\rangle_r$ in $\dot{H}^r(\R^d)$ is defined by \begin{align*}
\langle u,v\rangle_r:=\frac{1}{(2\pi)^d}\int_{\R^d}|\xi|^{2r}\hat{u}(\xi)\overline{\hat{v}(\xi)}d\xi.
\end{align*}
Furthermore, for $r\in (0,1)$, applying \cite[Proposition 1.37]{BCD} there exists $C_{d,r}>0$ such that \begin{align}\label{scalar product}
\langle u,v\rangle_r=C_{d,r}\int_{R^d}\int_{\R^d}(u(x)-u(y))(v(x)-v(y))|x-y|^{-d-2r}dxdy, \text{ for all }u,v\in  \dot{H}^r(\R^d).
\end{align}
If $r\in (0,d/2)$ then by \cite[Theorem 1.38]{BCD} there exists $S_{d,r}>0$ such that the fractional Sobolev inequality as\begin{align}\label{Sobolev ineq}
\|u\|_{L^q(\R^d)}\leq S_{d,r}\|u\|_{\dot{H}^r(\R^d)},
\end{align}
for every $u\in \dot{H}^r(\R^d)$ and here $q:=\frac{2d}{d-2r}$.

\section{JKO schemes and gradient flows in the modified Wasserstein space}
In this section, let $\mathbf{m}\in C^{\infty}(\R_+,\R_+)$ such that $\inf \bm>0$ and $\mathbf{m}''\leq 0$. We show the uniqueness of solutions of our JKO scheme and investigate gradient formulation of them.
\begin{theorem}
Let $u_0\in \cP_2(\R^d)\cap \dot{H}^{-s}(\R^d)$ then for every $\tau>0$, the scheme (\ref{JKO schemes}) has a unique solution.
\end{theorem}
\begin{proof}
It is clear that the map $u\mapsto \frac{C_\alpha}{2\tau^\alpha}\bW^2_\mathbf{m}(u,\overline{u}^{k-1}_\tau)+\frac{1}{2}\|u\|^2_{\dot{H}^{-s}(\R^d)}$ is bounded below. Moreover, by \cite[Theorem 5.5 and Theorem 5.6]{DNS} we get that it is lower-semicontinuous and has compact sub-levels under the weak* topology. Therefore, the scheme (\ref{JKO schemes}) has solutions. The uniqueness of the solution follows from the convexity of $\bW_\bm$ \cite[Theorem 5.11]{DNS} and the strict convexity of the map $u\mapsto \frac{1}{2}\|u\|^2_{\dot{H}^{-s}(\R^d)}$.
\end{proof}

Now we study gradient formulation of our scheme. Let $\varphi\in C_c^\infty (\R^d)$. Given $\delta>0$, let $S_\delta$ be the semigroup defined by $\mathbf{S}_{\delta,t}v_0=v_t$ for every $t>0$ with $v_t$ is the unique solution of the following equation with initial data $v_0\in \cP(\R^d)$
$$\partial_tv_t-\text{div}(\bm(v_t)\nabla \varphi)-\delta \Delta v_t=0\text{ in }(0,\infty)\times \R^d.$$
Let $\rho\in C_c^\infty \left([0,1]\times\R^d\right)$ such that $\rho(t)\in \cP(\R^d)$ for every $t\in [0,1]$. For every $h>0$ and $t>0$ we define $\rho^h(t):=\mathbf{S}_{\delta,ht}\rho(t)\in \cP(\R^d)$. Let $\phi^h$ be the unique solution of $$\partial_t\rho^h(t,x)=-\text{div}(\bm(\rho^h(t,x)))\vert \nabla \phi^h(t,x)\vert^2\text{ in }[0,1]\times \R^d.$$
For every $h,t>0$ we define $$\mathbf{A}^h(t):=\int_{\R^d}\bm(\rho^h(t,x))\vert \nabla \phi^h(t,x)\vert^2dx.$$
Next, let $U:[0,+\infty)\to [0,\infty)$ be the function defined by $U''(s):=1/\bm(s)$ with $U'(0)=U(0)=0$. Then we define $$\mathbf{U}(u):=\int_{\R^d}U(u(x))dx\text{ for every }u\in\cP(\R^d).$$
For any $\delta>0$, we define $\mathbf{V}_\delta:\cP(\R^d)\to \R$ by $$V_\delta (v):=\langle \varphi,v\rangle  +\delta \mathbf{U}(v)\text{ for every }v\in \cP(\R^d).$$

Let us recall several basic results of L1 scheme for discretizing the Caputo derivative $\partial^\alpha_t \varphi$, introduced by \cite{LX} and applied to solve time-fractional Fokker-Planck equations \cite{DJ}.

\begin{lemma}\label{L-sum of b_i}
\cite[Lemma 3.1]{DJ} For every $k\in \mathbb{N},k\geq 1$ and $0<\alpha <1$, we have $$\sum_{i=0}^{k-1}(-b^{(k)}_{k-i})=1\text{ and } \sum_{i=1}^k(-b^{(i)}_i)=k^{1-\alpha}.$$

\end{lemma}
\begin{lemma}\label{L-approximate fractional} Let $T>0$ and we consider $\tau=\frac{T}{N}$ with $N\in\N$ as a uniform partition of the interval $[0,T]$. Then for every test function $\varphi \in C^\infty_c(\mathbb{R}^d)$ and every $N\in \N$, we have \begin{enumerate}
	\item \cite[equations (3.12) and (3.13)]{LX}  \begin{align*}
\lim_{\tau\to 0}\frac{C_\alpha}{\tau^\alpha}\sum_{i=k}^Nb_{i-k}^{(N-k)}\varphi (t+(i-k)\tau)&={}_t D^\alpha_T\varphi (t),\,\forall t\in ((k-1)\tau, k\tau], k=1,\ldots, N.
\end{align*}
        \item \cite[Theorem 3.1]{DJ}  \begin{align*}
\lim_{\tau\to 0}\frac{C_\alpha}{\tau^\alpha}\sum_{k=1}^N b{(n)}_n\int_{(k-1)\tau}^{k\tau}\varphi(t)dt
&=-\frac{1}{\Gamma(1-\alpha)}\int_0^Tt^{-\alpha}\varphi(t)dt.\end{align*}
\end{enumerate}
\end{lemma}
\begin{lemma} \label{L-approximate fractional 1}  \cite[Lemma 3.3]{DJ} With the same assumptions as in Lemma \ref{L-approximate fractional}, for every $N\in \N$ and every $\psi\in C^1[0,T]$ such that $\psi(T)=0$, we have \begin{align*}
\sum_{k=1}^N \int_{(k-1)\tau}^{k\tau}\bigg(u^k_\tau -\sum_{i=1}^{k-1}(-b^{(k)}_{k-i}u^i_\tau) \bigg)\psi(t)dt= & \sum_{k=1}^N \int_{(k-1)\tau}^{k\tau}\hat{u}_\tau(t) \bigg(\sum_{i=k}^{N}b_{i-k}^{N-k}\psi(t+(i-k)\tau) \bigg)dt\\
 &+ u(0)\sum_{k=1}^N b_k^{(k)}\int_{(k-1)\tau}^{k\tau}\psi(t)dt.
\end{align*}
\end{lemma}
We also recall basic results of our modified Wasserstein distance and its flow interchange property.
	\begin{lemma}\cite[Lemma 2.1]{CN}
		\label{apro}
		Let $\mathbf{m}\in C^\infty (\R_+,\R_+)$ such that $\inf \bm>0$. Let $\mu^0,\mu^1\in \cP(\R^d)\cap L^1(\R^d)$ be such that $\mathbf{W}_{\bm}(\mu^0,\mu^1)<\infty$. Then there exist $\rho_n\in C^\infty_c([0,1]\times\R^d)$ and $\phi_n\in C^\infty([0,1]\times \R^d)\cap L^\infty([0,1],H^2(\R^d))$ such that 
		\begin{enumerate}	
			\item $\rho_n(t)\in \cP(\R^d)$ for every $ t\in[0,1] $, $||\rho_n(0)-\mu^0||_{L^1(\R^d)}+||\rho_n(1)-\mu^1||_{L^1(\R^d)}\to 0$ as $n\to \infty.$
			\item $\rho_n,\phi_n$ satisfies $\partial_t\rho_{n}(t,x)=-\operatorname{div}(\bm(\rho_n(t,x))\nabla\phi_n(t,x))$ and 
			\begin{align*}
				\mathbf{W}_{\bm}^2(\mu^0,\mu^1)=	\lim_{n\to +\infty}\int_{0}^{1}\int_{\R^d}\bm(\rho_n(t,x))|\nabla_x \phi_n(t,x)|^2 dxdt.
			\end{align*}
\end{enumerate}
	\end{lemma}

\begin{lemma}\label{Eulerian calculus} \cite[Lemma 3.3 and inequalities (3.14) and (4.4)]{CN}
Let $\mathbf{m}\in C^{\infty}(\R_+,\R_+)$ such that $\inf \bm>0$ and $\mathbf{m}''\leq 0$. 
\begin{enumerate}
	\item \label{L-Eulerian calculus} for every $t\in [0,1]$ and $h\geq 0$ one has $$\frac{1}{2}\partial_h\mathbf{A}^h(t)+\partial_t\mathbf{V}_{\delta}(\rho^h(t))\leq -\lambda_\delta\mathbf{A}^h(t),$$ where
\begin{align}\label{formula of delta}
\lambda_\delta =-\Vert D^2\varphi\Vert _{L^\infty}\sup_{z>0}\vert \bm'(z)\vert-\frac{1}{2\delta}\Vert \nabla \varphi\Vert^2_{L^\infty}\sup_{z>0}\left(\bm(z)\vert \bm''(z)\vert\right).\end{align}
\item \label{L-flows of modified equations1} for every  $ \xi,\mu\in\cP(\R^d) $ such that $ \mathbf{V}_\delta(\xi),\mathbf{V}_\delta(\mu),\mathbf{W}_\mathbf{m}\xi,\mu)<\infty$ then 
\begin{align*}
			\frac{1}{2}\limsup_{h\to 0}\frac{\mathbf{W}^2_{\bm}(\mathbf{S}_{\delta,h}(\xi),\mu)-\mathbf{W}^2_{\bm}(\xi,\mu)^2}{h}+\dfrac{\lambda_\delta}{2}\mathbf{W}^2_{\bm}(\mu,\xi)+ \mathbf{V}_\delta(\xi)\leq  \mathbf{V}_\delta(\mu).
		\end{align*}
\item \label{F-convergence of solution W_m}  for every $\mu,\xi\in\cP(\R^d)$ such that $\mathbf{U}(\mu),\mathbf{U}(\xi)<\infty$ and $\mathbf{W}_\mathbf{m}(\mu,\xi)<\infty$ we have \begin{align*}
			\limsup_{h\to 0}\dfrac{\mathbf{W}^2_\mathbf{m}(\mathbf{H}_h\xi,\mu)-\mathbf{W}^2_\mathbf{m}(\xi,\mu)}{2h}\leq \mathbf{U}(\mu)-\mathbf{U}(\xi),
		\end{align*}
where $\mathbf{H}_t$ is the semigroup with respect to the heat equation $v_t=\Delta v$ in $\R^d$ with initial data $v_0\in \cP(\R^d)$.

\end{enumerate}

\end{lemma}
\begin{remark}
In  \cite[Lemma 3.3 and inequalities (3.14) and (4.4)]{CN}, authors considered the function $\mathbf{m}$ is given by $\mathbf{m}(z)=(z+\tau^{1/10})^{\alpha}$ ($0<\alpha\leq 1$) and $\mathbf{m}(z)=(z+1)^{\tau^{1/10}}$. However, the proof there also works for every $\mathbf{m}\in C^{\infty}(\R_+,\R_+)$ such that $\inf \bm>0$ and $\mathbf{m}''\leq 0$. 
\end{remark}

\begin{lemma}
\label{L-flows of modified equations}
Let $\lambda_\delta$ as in Lemma \ref{Eulerian calculus}, $\{u^k_\tau\}_{k\in\N}$ be the solution of \eqref{JKO schemes} and $\varphi \in C^\infty_c(\mathbb{R})$ be a test function then \begin{align*}
			\frac{C_\alpha}{\tau^\alpha}\left(\mathbf{V}_\delta(u^k_\tau)-\mathbf{V}_\delta(\overline{u}^{k-1}_\tau)\right)\leq -\delta||u^k_\tau||_{\dot{H}^{1-s}}^2+\langle \operatorname{div} \left(\bm(u^k_\tau)\nabla (-\Delta) ^{-s}u^k_\tau\right),\varphi \rangle-\dfrac{\lambda_\delta C_\alpha}{2\tau^\alpha}\mathbf{W}^2_{\bm}(u^k_\tau,\overline{u}^{k-1}_\tau),
		\end{align*}
for every $\tau>0$ and every $k\in\N$.\end{lemma}
\begin{proof}
We recall that the semigroup $ \mathbf{S}_\delta $ is given by $ \mathbf{S}_{\delta,t}v_0=v_t, $ where $ v $ is the solution of the following equation with initial data $v_0\in \cP(\R^d)$ \begin{align*}
			\partial_tv-\operatorname{div}(\bm(v)\nabla\varphi)-\delta\Delta v=0.
		\end{align*}
By the definition of $ u_\tau^k, $ we have that \begin{align*}
			\frac{C_\alpha}{2\tau^\alpha}\bW^2_\mathbf{m}(u^k_\tau,\overline{u}^{k-1}_\tau)+\frac{1}{2}\Vert u^k_\tau\Vert^2_{\dot{H}^{-s}(\R^d)}\leq\frac{C_\alpha}{2\tau^\alpha}\bW^2_\mathbf{m}(\mathbf{S}_{\delta,h}u^k_\tau,\overline{u}^{k-1}_\tau)+\frac{1}{2}\Vert \mathbf{S}_{\delta,h}u^k_\tau\Vert^2_{\dot{H}^{-s}(\R^d)}.
		\end{align*}
This implies that \begin{align}\label{Euler1}
\frac{C_\alpha}{2\tau^\alpha}\limsup_{h\to 0}\frac{\bW^2_\mathbf{m}(u^k_\tau,\overline{u}^{k-1}_\tau)-\bW^2_\mathbf{m}(\mathbf{S}_{\delta,h}u^k_\tau,\overline{u}^{k-1}_\tau))}{h}\leq \frac{1}{2}\partial_h\big|_{h=0}\Vert \mathbf{S}_{\delta,h}u^k_\tau\Vert^2_{\dot{H}^{-s}(\R^d)}.
\end{align}		
		Next, we have
		\begin{align}\label{Euler2}
			\partial_h\big|_{h=0}\Vert \mathbf{S}_{\delta,h}u^k_\tau\Vert^2_{\dot{H}^{-s}(\R^d)}=&2\int_{\R^d}|\xi|^{-2s}\overline{\widehat{u^k_\tau}}\partial_h\big|_{h=0}\widehat{\mathbf{S}_{\delta,h}u^k_\tau}d\xi\notag\\
			=&2\int_{\R^d}|\xi|^{-2s}\overline{\widehat{u^k_\tau}}\partial_h\big|_{h=0}\left(\int_{\R^d}e^{-ix\xi}\mathbf{S}_{\delta,h}u^k_\tau dx\right)d\xi\notag\\
			=&2\int_{\R^d}|\xi|^{-2s}\overline{\widehat{u^k_\tau}}\left(\int_{\R^d}e^{-ix\xi} \operatorname{div}(\bm(u^k_\tau)\nabla\varphi) dx+\delta\int_{\R^d}e^{-ix\xi} \Delta u^k_\tau dx\right)d\xi\notag\\
			=&2\int_{\R^d}|\xi|^{-2s}\overline{\widehat{u^k_\tau}}\left(\widehat{ \operatorname{div}(\bm(u^k_\tau)\nabla\varphi)}+\delta\widehat{\Delta u^k_\tau} dx\right)\notag\\
			=&2\langle \operatorname{div} \left(\bm(u^k_\tau)\nabla (-\Delta) ^{-s}u^k_\tau\right),\varphi \rangle-2\delta||u^k_\tau||_{\dot{H}^{1-s}}^2.
		\end{align}
Now, using \eqref{Euler1}, \eqref{Euler2} and Lemma \ref{Eulerian calculus}(\ref{L-flows of modified equations1}) we get the result.
\end{proof}

\begin{lemma}\label{L-estimate of u bar}
Let $F:\cP(\R^d)\to [0,\infty)$ be a convex function. For $T>0$, we set $\tau=\frac{T}{N}$ with $N\in \N$ and let $\{u^k_\tau\}_{k\in\N}$ be the solution of \eqref{JKO schemes}. Then for every $N\in \N$ and every $k=1,\ldots, N $ we have \begin{align*}
\sum_{i=1}^k \left( F(\overline{u}^k_\tau)-F(u^k_\tau)\right)\leq \left(\frac{T}{\tau} \right)^{1-\alpha}F(u_0).
\end{align*}
\end{lemma}
\begin{proof}
For every $1\leq i\leq k$, by the definition of $\overline{u}^{i-1}$ and the convexity of $F$ one has \begin{align*}
F(\overline{u}^{i-1}_\tau)\leq \sum_{j=0}^{i-1}(-b^{(i)}_{i-j})F(u^j_\tau), 
\end{align*}
since $\sum_{j=0}^{i-1}(-b^{(i)}_{i-j})=1$ (Lemma \ref{L-sum of b_i}). Therefore,
\begin{align*}
\sum_{i=1}^k\left(F(\overline{u}^{i-1}_\tau)-F(u^i_\tau)\right)
& \leq \sum_{i=1}^k\sum_{j=0}^{i-1}(-b^{(i)}_{i-j})F(u^j_\tau)-\sum_{i=1}^k F(u^i_\tau) \\
& = \sum_{i=1}^k(-b_i^{(i)})F(u_0)+\sum_{i=1}^{k-1}\left(\sum_{j=1}^{k-i}(-b_j^{(j+i)})-1\right)F(u^i_\tau)-F(u^k_\tau).
\end{align*}
Using Lemma \ref{L-sum of b_i}, we get that $\sum_{i=1}^k(-b_i^{(i)})=k^{1-\alpha}$ and notice that $\sum_{j=1}^{k-i}(-b_j^{(j+i)})\leq 1$ for every $i\in\{2,\ldots,k-1\}$. Hence, we obtain that \begin{align*}
\sum_{i=1}^k\left( F(\overline{u}^{i-1}_\tau)-F(u^i_\tau)\right) & \leq k^{1-\alpha} F(u_0)-F(u^k_\tau)\\
& \leq N^{1-\alpha} F(u_0)\\
&=\bigg(\frac{T}{\tau}\bigg)^{1-\alpha} F(u_0).
\end{align*}
Thus, we get the result.
\end{proof}

\begin{lemma}
\label{L-convergence of u}
Let $u_0\in \cP(\R^d)\cap (L^2\cap \dot{H}^{-s})(\R^d)$ and $\{u^k_\tau\}_{k\in\N}$ be the solution of \eqref{JKO schemes}. Then for every $\tau>0$ and every $k\in \N$ one has \begin{align} \label{F-convergence of u}
			\Vert u^k_\tau\Vert^2_{\dot{H}^{1-s}(\R^d)}\leq\frac{C_\alpha}{\tau^\alpha}\left( \mathbf{U}(\overline{u}^{k-1}_\tau)-\mathbf{U}(u^k_\tau)\right).
		\end{align}
Moreover, for every $T>0$ we have  $\hat{u}_{\tau}\to u$ weakly in $L^2((0,T); \dot{H}^{1-s}(\R^d))$ as $\tau\to 0$ for some $u:[0,\infty)\to \cP(\R^d)$.		
\end{lemma}
\begin{proof}
We recall that $\mathbf{H}_t$ is the semigroup with respect to the heat equation $u_t=\Delta u$ in $\R^d$ with initial data $u_0$, then by the scheme \eqref{JKO schemes} we get that 
		\begin{align*}
			\frac{C_\alpha}{2\tau^\alpha}\mathbf{W}^2_\mathbf{m}(u^k_\tau,\overline{u}^{k-1}_\tau)+\frac{1}{2}\Vert u^k_\tau\Vert^2_{\dot{H}^{-s}(\R^d)}\leq\frac{C_\alpha}{2\tau^\alpha}\mathbf{W}^2_\mathbf{m}(\mathbf{H}_hu^k_\tau,\overline{u}^{k-1}_\tau)+\frac{1}{2}\Vert \mathbf{H}_hu^k_\tau\Vert^2_{\dot{H}^{-s}(\R^d)}.
		\end{align*}
		
	Next, we will calculate the derivative of $ \Vert \mathbf{H}_hu^k_\tau\Vert^2_{\dot{H}^{-s}(\R^d)} $ at $ h=0 $.
		\begin{align}\label{F-convergence of solution derivative of H}
			\partial_h\big|_{h=0}\Vert \mathbf{H}_hu^k_\tau\Vert^2_{\dot{H}^{-s}(\R^d)}=&2\int_{\R^d}|\xi|^{-2s}\overline{\widehat{u^k_\tau}}\partial_h\big|_{h=0}\widehat{\mathbf{H}_hu^k_\tau}d\xi
			\notag\\=&2\int_{\R^d}\overline{\widehat{(-\Delta)^{-s}u^k_\tau}}\partial_h\bigg|_{h=0}\left(\int_{\R^d}e^{-ix\xi}\mathbf{H}_hu^k_\tau dx\right)d\xi\notag\\
			=&2\int_{\R^d}\overline{\widehat{(-\Delta)^{-s}u^k_\tau}}\widehat{\Delta u^k_\tau} d\xi\notag\\
			=&-2\Vert u^k_\tau\Vert^2_{\dot{H}^{1-s}(\R^d)}.
		\end{align}
		Combining Lemma \ref{Eulerian calculus} \eqref{F-convergence of solution W_m} and \eqref{F-convergence of solution derivative of H} we get (\ref{F-convergence of u}).
		
		 Now, for each $T>0$ we consider a uniform partition of the interval $[0,T]$ as $\tau=\frac{T}{N}$ for $N\in \N$, we will prove that \begin{align}\label{F-bounded of U 1}
\int_0^{k\tau}\Vert \hat{u}_{\tau}(t)\Vert^2_{\dot{H}^{1-s}}dt \leq C_\alpha T^{1-\alpha} \mathbf{U}(u_0).
\end{align}
Indeed, as $\mathbf{U}$ is convex, by using Lemma \ref{L-estimate of u bar}, one gets that 
\begin{align} \label{F-bounded of U 2}
\sum_{i=1}^k\left( \mathbf{U}(\overline{u}^{i-1}_\tau)-\mathbf{U}(u^i_\tau)\right) \leq \bigg(\frac{T}{\tau}\bigg)^{1-\alpha} \mathbf{U}(u_0).
\end{align}
So, the inequality (\ref{F-bounded of U 1}) follows (\ref{F-convergence of u}) and (\ref{F-bounded of U 2}). Thus, $\int_0^T\|\hat{u}_{\tau}(t)\|^2_{\dot{H}^{1-s}}dt$ is bounded. This implies that there exists $u:[0,\infty)\to \cP(\R^d)$ such that $\hat{u}_{\tau}$ converges weakly to $u$ in $L^2((0,T); \dot{H}^{1-s}(\R^d))$ as $\tau\to 0$.
\end{proof}

\begin{lemma}\label{L-estimate of solution}
Let $u_0\in \cP(\R^d)\cap (L^2\cap \dot{H}^{-s})(\R^d)$ and $\{u^k_\tau\}_{k\in\N}$ be the solution of \eqref{JKO schemes}. Then for every $\tau, T>0$, every $\varphi\in C_c^\infty(\R^d)$, and every $\psi\in C^\infty([0,T],\R_+)$ with $\psi(T)=0$ we get that \begin{enumerate}
	\item  there exists a constant $C(\psi,\varphi)>0$ such that \begin{align*}
&\int_{0}^{T}\psi(t)\langle \textup{div}\left(\left(\hat{u}_{\tau}(t)+\tau^{\frac{\alpha}{4(2-\beta)}}\right)^{\beta}(\nabla(-\Delta)^{-s}\hat{u}_{\tau}(t)\right),\varphi\rangle dt\\
\leq& C(\psi,\varphi) \tau^{\frac{\alpha\beta}{4(2-\beta)}}+\int_{0}^{T}\psi(t)\langle \textup{div}\left(\hat{u}_{\tau}^{\beta}(t)(\nabla(-\Delta)^{-s}\hat{u}_{\tau}(t)\right),\varphi\rangle dt.
\end{align*}
         \item for $\tau>0$ is small enough, there exists a constant $R(\psi,\varphi)>0$ such that
         \begin{align*}
&\int_{0}^{T}\psi(t)\langle \textup{div}\left((\hat{u}_{\tau}(t)+1)^{\tau^{1-\alpha/4}}(\nabla(-\Delta)^{-s}\hat{u}_{\tau}(t)\right),\varphi\rangle dt\\
\leq& \tau^{1-\alpha/4}R(\psi,\varphi)-\int^{T}_0\psi(t)\langle (-\Delta)^{1-s}\hat{u}_{\tau}(t),\varphi\rangle dt.
\end{align*}

\end{enumerate}\end{lemma}
\begin{proof}
$1.$ For every $r>0$, one has the function $q(a)=(a+r)^{\beta}-a^{\beta}$ is non-increasing on $(0;+\infty)$. Therefore, we get that
$$0\leq (a+r)^{\beta}-a^{\beta}\leq r^{\beta},\text{ for every }a,r>0.$$
From this inequality we have
\begin{align*}
			&\int_{0}^{T}\psi(t)\langle \textup{div}\left(\left(\hat{u}_{\tau}(t)+\tau^{\frac{\alpha}{4(2-\beta)}}\right)^{\beta}\nabla(-\Delta)^{-s}\hat{u}_{\tau}(t)\right),\varphi\rangle dt\\
			\leq&\int_{0}^{T}\psi(t)\langle \textup{div}\left(\hat{u}_{\tau}^{\beta}(t)(\nabla(-\Delta)^{-s}\hat{u}_{\tau}(t)\right),\varphi\rangle dt\\
			&+\int_{0}^{T}\int_{\R^d}\psi(t) \left\vert \left(\hat{u}_{\tau}(t)+\tau^{\frac{\alpha}{4(2-\beta)}}\right)^{\beta}-\hat{u}^{\beta}_{\tau}(t)\right\vert\vert\nabla(-\Delta)^{-s}\hat{u}_{\tau}(t)\vert \vert \nabla \varphi\vert dxdt\\
			\leq & \int_{0}^{T}\psi(t)\langle \textup{div}\left(\hat{u}_{\tau}^{\beta}(t)(\nabla(-\Delta)^{-s}\hat{u}_{\tau}(t)\right),\varphi\rangle dt + \tau^{\frac{\alpha\beta}{4(2-\beta)}}\int_{0}^{T}\int_{\R^d}\psi(t) \vert\nabla(-\Delta)^{-s}\hat{u}_{\tau}(t)\vert \vert \nabla \varphi\vert dxdt.		\end{align*}
Since $\varphi\in C^\infty_c(\R^d)$ and $\psi\in C^\infty ([0,T],\R_+)$ one gets that there exists a constant $C_1(\psi,\varphi)>0$ such that
		\begin{align*}
			\int_0^T\int_{\R^d}\psi(t) \vert\nabla(-\Delta)^{-s}\hat{u}_{\tau}(t)\vert \vert \nabla \varphi\vert dxdt
			&\leq C_1(\psi,\varphi)\int_0^T\bigg(\int_{\R^d}\vert \nabla(-\Delta)^{-s}\hat{u}_{\tau}(t)\vert^2 dx\bigg)^{1/2}dt\\
			& =C_1(\psi,\varphi)\int_0^T\bigg(\int_{\R^d}(-\Delta)^{1-s}\hat{u}_{\tau}(t)(-\Delta)^{-s}\hat{u}_{\tau}(t)dx\bigg)^{1/2}dt\\
			&= C_1(\psi,\varphi)\int_0^T \Vert \hat{u}_{\tau}(t)\Vert_{\dot{H}^{1-2s}(\R^d)}dt.
		\end{align*} 
By interpolation inequality we get that \begin{align*}
\Vert \hat{u}_{\tau}(t)\Vert_{\dot{H}^{1-2s}(\R^d)}\leq  \Vert \hat{u}_{\tau}(t)\Vert_{\dot{H}^{-s}(\R^d)}^s\Vert \hat{u}_{\tau}(t)\Vert_{\dot{H}^{1-s}(\R^d)}^{1-s}.
\end{align*}
Since $\Vert \hat{u}_{\tau}(t)\Vert_{\dot{H}^{1-s}(\R^d)}$ and $\Vert \hat{u}_{\tau}(t)\Vert_{\dot{H}^{-s}(\R^d)}$	are bounded, there exists $C(\psi,\varphi)>0$ such that \begin{align*}
\int_0^T\int_{\R^d}\psi(t) \vert\nabla(-\Delta)^{-s}\hat{u}_{\tau}(t)\vert \vert \nabla \varphi\vert dxdt\leq C(\psi,\varphi).
\end{align*}	
Hence, we get the result.\\\\
$2.$ Since the function $w(a)=r(a+1)^r\log(a+1)-(a+1)^r+1$ is non-decreasing on $(0;+\infty)$ for every $r>0$ one gets that $$0\leq (a+1)^r-1\leq r(a+1)^r\log(a+1),\text{ for every }a,r>0.$$
So that
\begin{align*}
			&\int_{0}^{T}\psi(t)\langle \text{div}\left((\hat{u}_{\tau}(t)+1)^{\tau^{1-\alpha/4}}(\nabla(-\Delta)^{-s}\hat{u}_{\tau}(t)\right),\varphi\rangle dt\\
			\leq&\tau^{1-\alpha/4} \int_{0}^{T}\int_{\R^d}\psi(t)(\hat{u}_{\tau}(t)+1)^{\tau^{1-\alpha/4}}\log(\hat{u}_{\tau}(t)+1)\vert \nabla(-\Delta)^{-s}\hat{u}_{\tau}(t)\vert \vert \nabla \varphi\vert dx dt\\
			&-\int_{0}^{T}\psi(t)\langle (-\Delta)^{1-s}\hat{u}_{\tau}(t),\varphi\rangle dt.
		\end{align*}
Similarly as above, there exists a constant $R_1(\psi,\varphi)>0$ such that for $\tau>0$ is small enough,
		\begin{align*}
			&\int_{0}^{T}\int_{\R^d}\psi(t)(\hat{u}_{\tau}(t)+1)^{\tau^{1-\alpha/4}}\log(\hat{u}_{\tau}(t)+1)\vert \nabla(-\Delta)^{-s}\hat{u}_{\tau}(t)\vert \vert \nabla \varphi\vert dx dt\\
			&\leq R_1(\psi,\varphi)\int_0^T\int_{\text{supp}\varphi}(\hat{u}_{\tau}(t)+1)^{\tau^{1-\alpha/4}}\log(\hat{u}_{\tau}(t)+1)\vert \nabla(-\Delta)^{-s}\hat{u}_{\tau}(t)\vert dxdt.
		\end{align*}
		Furthermore, from the inequality $\log(r+1)\leq 2r^{1/3}$ for every $r\geq 0$ we also obtain that for $\tau>0$ is small enough, \begin{align*}
			(\hat{u}_{\tau}(t)+1)^{\tau^{1-\alpha/4}}\log(\hat{u}_{\tau}(t)+1)&\leq 2(\hat{u}_{\tau}(t)+1)^{\tau^{1-\alpha/4}+1/3}\leq 2(\hat{u}_{\tau}(t)+1)^{1/2}. 
		\end{align*}
		Therefore, observe that $\psi\in C^\infty([0,T],\R_+)$  and $\varphi\in C_c^\infty(\R^d)$ we get that there exists constants $R_2(\psi,\varphi),R(\psi,\varphi)>0$ such that for $\tau>0$ is small enough,
		\begin{align*}
			&\int_{0}^{T}\int_{\R^d}\psi(t)(\hat{u}_{\tau}(t)+1)^{\tau^{1-\alpha/4}}\log(\hat{u}_{\tau}(t)+1)\vert \nabla(-\Delta)^{-s}\hat{u}_{\tau}(t)\vert \vert \nabla \varphi\vert dx dt\\
			\leq &R_2(\psi,\varphi)\int_0^T\int_{\text{supp}\varphi}(\hat{u}_{\tau}(t)+1)dxdt+R_2(\psi,\varphi)\int_0^T\int_{\R^d}\vert \nabla(-\Delta)^{-s}\hat{u}_{\tau}(t)\vert^2 dxdt\\
			=& R_2(\psi,\varphi)\int_0^T\int_{\text{supp}\varphi}(\hat{u}_{\tau}(t)+1)dxdt+R_2(\psi,\varphi)\int_0^T \Vert \hat{u}_{\tau}(t)\Vert_{\dot{H}^{1-2s}(\R^d)}^2dt\\
			\leq& R(\psi,\varphi).
		\end{align*} 
Thus, we get the result.				 				
\end{proof}

Next, we establish several estimates for our regularizing effect of the interpolation approximation $\hat{u}_{\tau}(t)$. Let $g\in C^2([0,\infty),\R_+)$ be a convex function such that $g(0)=g'(0)=g''(0)=0$. For any $\delta>0$, we define $\mathbf{U}_\delta:\cP_2(\R^d)\to (-\infty,+\infty]$ by $$\mathbf{U}_\delta(u):=\int_{\R^d}g(u(x))dx+\delta\mathbf{U}(u).$$

We define the semigroup $\mathbf{K}_\delta$ by $\mathbf{K}_{\delta,t}v_0=v_t$ for every $t>0$ with $v_t$ is the unique solution of the following equation with initial data $v_0\in \cP(\R^d)$
	$$\partial_tv_t-\Delta \mathbf{G}(v_t)-\delta\Delta v_t=0 \text{ in }(0,+\infty)\times \R^d,$$
	where $\mathbf{G}(r)=\int_0^r\bm(z)g''(z)dz$.
	\begin{lemma}\cite[the inequality (4.11)]{CN} \label{L-flows of modified equations for U_delta}
		For every $\mu,\xi\in\cP(\R^d)$ such that $\mathbf{U}_\delta(\mu),\mathbf{U}_\delta(\xi)<\infty$ and $\mathbf{W}_\mathbf{m}(\mu,\xi)<\infty$ we have $$\limsup_{h\to 0}\dfrac{\mathbf{W}^2_\mathbf{m}(\mathbf{K}_{\delta,h}(\xi),\mu)-\mathbf{W}^2_\mathbf{m}(\xi,\mu)}{2h}\leq \mathbf{U}_\delta (\mu)-\mathbf{U}_\delta(\xi).$$
	\end{lemma}	
\begin{lemma}\label{L-estimate solution}
		Let $u_0\in \cP_2(\R^d)\cap (L^2\cap \dot{H}^{-s})(\R^d)$ and $\{u^k_\tau\}_{k\in\N}$ be the solution of \eqref{JKO schemes}.  Let $g\in C^2([0,\infty),\R_+)$ be a convex function such that $g(0)=g'(0)=g''(0)=0$.  Then for every $\tau>0$ and every $k\in\N$ we have
		\begin{align*}
		\Vert \mathcal{G}(u^k_\tau)\Vert^2_{\dot{H}^{1-s}(\R^d)}\leq \frac{C_\alpha}{\tau^\alpha}\left(\int_{\R^d} g(\overline{u}^{k-1}_\tau(x))dx-\int_{\R^d} g(u^{k}_\tau(x))dx\right),
		\end{align*} and
		$$ \Vert \mathcal{G}(u^k_\tau)\Vert^2_{L^{\frac{2d}{d-2(1-s)}}(\R^d)}\leq \frac{C_\alpha S_{d,1-s}^2}{\tau^\alpha}\left(\int_{\R^d} g(\overline{u}^{k-1}_\tau(x))dx-\int_{\R^d} g(u^{k}_\tau(x))dx\right),$$
		where $\mathcal{G}(r)=\int_0^r\sqrt{\bm(z)g''(z)}dz$ and $S_{d,1-s}$ is determined as in \eqref{Sobolev ineq}
	\end{lemma}
	\begin{proof}
		First, we have 
		\begin{align*}
			\partial_h\big|_{h=0}\Vert \mathbf{K}_{\delta,h}u^k_\tau\Vert^2_{\dot{H}^{-s}(\R^d)}=&2\int_{\R^d}|\xi|^{-2s}\overline{\widehat{u^k_\tau}}\partial_h\big|_{h=0}\widehat{\mathbf{K}_{\delta,h}u^k_\tau}d\xi\\=&2\int_{\R^d}\overline{\widehat{(-\Delta)^{-s}u^k_\tau}}\partial_h\big|_{h=0}\left(\int_{\R^d}e^{-ix\xi}\mathbf{K}_{\delta,h}u^k_\tau dx\right)d\xi\\
			=&2\int_{\R^d}\overline{\widehat{(-\Delta)^{-s}u^k_\tau}}\left(\widehat{\Delta \mathbf{G}(u^k_\tau)}+\delta \widehat{\Delta u^k_\tau} \right) d\xi\\
			=&-2\langle u^k_\tau, (-\Delta)^{1-s}\mathbf{G}(u^k_\tau)\rangle-2\delta\Vert u^k_\tau\Vert^2_{\dot{H}^{1-s}(\R^d)}.
		\end{align*}
Next, by using the scheme (\ref{JKO schemes}) we obtain that
\begin{align*}
			\frac{C_\alpha}{2\tau^\alpha}\bW^2_\mathbf{m}(u^k_\tau,\overline{u}^{k-1}_\tau)+\frac{1}{2}\Vert u^k_\tau\Vert^2_{\dot{H}^{-s}(\R^d)}\leq\frac{C_\alpha}{2\tau^\alpha}\bW^2_\mathbf{m}(\mathbf{K}_{\delta,h}u^k_\tau,\overline{u}^{k-1}_\tau)+\frac{1}{2}\Vert \mathbf{K}_{\delta,h}u^k_\tau\Vert^2_{\dot{H}^{-s}(\R^d)}.
		\end{align*}
This yields, \begin{align*}
\frac{C_\alpha}{2\tau^\alpha}\limsup_{h\to 0}\frac{\bW^2_\mathbf{m}(u^k_\tau,\overline{u}^{k-1}_\tau)-\bW^2_\mathbf{m}(\mathbf{K}_{\delta,h}u^k_\tau,\overline{u}^{k-1}_\tau))}{h}\leq \frac{1}{2}\partial_h\big|_{h=0}\Vert \mathbf{K}_{\delta,h}u^k_\tau\Vert^2_{\dot{H}^{-s}(\R^d)}.
\end{align*}			
Therefore, using Lemma \ref{L-flows of modified equations for U_delta} we get that $$\langle u^k_\tau, (-\Delta)^{1-s}\mathbf{G}(u^k_\tau)\rangle + \delta\Vert u^k_\tau\Vert^2_{\dot{H}^{1-s}(\R^d)}\leq \frac{C_\alpha}{\tau^\alpha}\left[\mathbf{U}_\delta(\overline{u}^{k-1}_\tau)-\mathbf{U}_\delta(u^{k}_\tau)\right].$$
		Taking $\delta\to 0$, one gets that \begin{align*}
\langle \langle u^k_\tau, (-\Delta)^{1-s}\mathbf{G}(u^k_\tau)\rangle\rangle  \leq \frac{C_\alpha}{\tau^\alpha}\left(\int_{\R^d} g(\overline{u}^{k-1}_\tau(x))dx-\int_{\R^d} g(u^{k}_\tau(x))dx\right).
\end{align*}
		As $\mathbf{G}$ is non-decreasing, using \cite[Proposition 2.2]{LMS} we also have that $$\langle u^k_\tau, (-\Delta)^{1-s}\mathbf{G}(u^k_\tau)\rangle\geq 0.$$
		Notice that for any $a\leq b$, using Holder's inequality we have $$\left(\mathcal{G}(b)-\mathcal{G}(a)\right)^2=\left(\int_a^b\mathcal{G}'(z)dz\right)^2\leq (b-a)\int_a^b(\mathcal{G}'(z))^2dz=(b-a)(\mathbf{G}(b)-\mathbf{G}(a)).$$ 
		From this inequality and \eqref{scalar product} we get that \begin{align*}
			&\Vert \mathcal{G}(u^k_\tau)\Vert^2_{\dot{H}^{1-s}(\R^d)}\\
			&=C_{d,1-s}\int_{\R^d\times \R^d}\left(\mathcal{G}(u^k_\tau (x))-\mathcal{G}(u^k_\tau(y))\right)^2\vert x-y\vert ^{-d-2(1-s)}dxdy\\
			&\leq C_{d,1-s}\int_{\R^d\times \R^d}\left(u^k_\tau(x)-u^k_\tau(y)\right)\left(\mathbf{G}(u^k_\tau(x))-\mathbf{G}(u^k_\tau(y))\right)\vert x-y\vert ^{-d-2(1-s)}dxdy\\
			&=\langle u^k_\tau, \mathbf{G}(u^k_\tau)\rangle_{1-s}\\
			&=\langle u^k_\tau, (-\Delta)^{1-s}\mathbf{G}(u^k_\tau)\rangle. 
		\end{align*}
Moreover, as $1-s\in(0,d/2)$, using fractional Sobolev inequality one gets that \begin{align*}
\Vert \mathcal{G}(u^k_\tau)\Vert^2_{L^{\frac{2d}{d-2(1-s)}}(\R^d)}\leq S^2_{d,1-s}\Vert \mathcal{G}(u^k_\tau)\Vert^2_{\dot{H}^{1-s}(\R^d)}.
\end{align*} 		
Thus, we get the result.
\end{proof}

\section{Proof of main results}
\begin{proof}[Proof of Theorem \ref{main result}]$1.$ This follows from Lemma \ref{L-convergence of u}.\\\\
$2.$ Fix $T>0$, we consider a uniform partition of the time interval $[0,T]$ with $\tau=\frac{T}{N}$ for $N\in\N$. Let any $\psi\in C^\infty([0,T],\R_+)$ with $\psi(T)=0$. 
Since \eqref{formula of delta} and observe that $\mathbf{m}(z)=\left(z+\tau^{\frac{\alpha}{4(2-\beta)}}\right)^{\beta}$ we have that 
\begin{align*}
-\lambda_\delta =\beta\|D^2\varphi\|_{L^\infty}\sup_{z>0}\left(z+\tau^{\frac{\alpha}{4(2-\beta)}}\right)^{\beta-1}+\frac{\beta( 1-\beta)}{2\delta}\|\nabla \varphi\|^2_{L^\infty}\sup_{z>0}\left(z+\tau^{\frac{\alpha}{4(2-\beta)}}\right)^{2(\beta-1)}.
\end{align*}
Therefore, there exists $C_1(\varphi)>0$ such that for $\tau>0$ is small enough $$-\lambda_\delta\leq C_1(\varphi)\left(\tau^{\frac{\alpha (\beta-1)}{4(2-\beta)}}+\frac{1}{\delta \tau^{\frac{\alpha (1-\beta)}{2(2-\beta)}}}\right).$$
By \eqref{JKO schemes}, for every $k=1,\ldots, N$, one has \begin{align*}
0\leq \frac{C_{\alpha}}{2\tau^\alpha}\mathbf{W}_{\mathbf{m}}^2(u^k_{\tau}, \overline{u}^{k-1}_{\tau})\leq B_k,
\end{align*}
where $$B_k:=\frac{1}{2}\Vert \overline{u}^{k-1}_{\tau} \Vert^2_{\dot{H}^{-s}}-\frac{1}{2}\Vert u^k_{\tau}\Vert^2_{\dot{H}^{-s}}.$$ 
Then by Lemma \ref{L-flows of modified equations} we get that \begin{align*}
&\sum_{k=1}^N\frac{C_\alpha}{\tau^\alpha}\int_{(k-1)\tau}^{k\tau}\psi(t)\left(\mathbf{V}_{\tau^{\alpha/(4-2\beta)}}(u^k_{\tau})-\mathbf{V}_{\tau^{\alpha/(4-2\beta)}}(\overline{u}^{k-1}_{\tau})\right)dt\\
 \leq & \sum_{k=1}^N\int_{(k-1)\tau}^{k\tau}\psi(t)\langle \text{div}\left(\left(u^k_{\tau}+\tau^{\frac{\alpha}{4(2-\beta)}}\right)^{\beta}(\nabla(-\Delta)^{-s}u^k_{\tau}\right),\varphi\rangle dt\\
			&+C_1(\varphi)\left(\tau^{\frac{\alpha (\beta-1)}{4(2-\beta)}}+\tau^{-\alpha/2}\right)\sum_{k=1}^NB_k \int_{(k-1)\tau}^{k\tau}\psi(t)dt.
\end{align*}
As  $\psi\in C^\infty([0,T],\R_+)$, there exists $C_2(\psi)>0$ such that $$\sum_{k=1}^NB_k \int_{(k-1)\tau}^{k\tau}\psi(t)dt\leq C_2(\psi)\tau\sum_{k=1}^NB_k.$$
Moreover, observe that $\|\cdot\|^2_{\dot{H}^{-s}}$ is convex, applying Lemma \ref{L-estimate of u bar}, we have 
\begin{align*}
			\sum_{k=1}^N\left(\Vert \overline{u}^{k-1}_{\tau} \Vert^2_{\dot{H}^{-s}}-\Vert u^k_{\tau}\Vert^2_{\dot{H}^{-s}}\right)\leq \left(\frac{T}{\tau}\right)^{1-\alpha}\|u_0\|^2_{\dot{H}^{-s}}.
					\end{align*}  
Thus,  \begin{align*}
\sum_{k=1}^NB_k\int_{(k-1)\tau}^{k\tau}\psi(t)dt\leq\frac{1}{2}C_2(\psi)\tau^{\alpha}T^{1-\alpha}\|u_0\|^2_{\dot{H}^{-s}}.
\end{align*}		
From this inequality and Lemma \ref{L-estimate of solution} we get that \begin{align}\label{F-main result}
&\sum_{k=1}^N\frac{C_\alpha}{\tau^\alpha}\int_{(k-1)\tau}^{k\tau}\psi(t)\left(\mathbf{V}_{\tau^{\alpha/(4-2\beta)}}(u^k_{\tau})-\mathbf{V}_{\tau^{\alpha/(4-2\beta)}}(\overline{u}^{k-1}_{\tau})\right)dt\notag\\
\leq & C(\psi,\varphi) \tau^{\frac{\alpha\beta}{4(2-\beta)}}+\int_{0}^{T}\psi(t)\langle \textup{div}\left(\hat{u}_{\tau}^{\beta}(t)(\nabla(-\Delta)^{-s}\hat{u}_{\tau}(t)\right),\varphi\rangle dt\notag\\
&+\frac{1}{2}C_1(\varphi)C_2(\psi)T^{1-\alpha}\left(\tau^{\frac{\alpha (7-3\beta)}{4(2-\beta)}}+\tau^{\alpha/2}\right)\|u_0\|^2_{\dot{H}^{-s}}.
\end{align}
Next, we will calculate the LHS of \eqref{F-main result}. By definition of $\mathbf{V}_\delta$ one has that \begin{align}\label{LHS-main result}
&\sum_{k=1}^N\frac{C_\alpha}{\tau^\alpha}\int_{(k-1)\tau}^{k\tau}\psi(t)\left(\mathbf{V}_{\tau^{\alpha/(4-2\beta)}}(u^k_{\tau})-\mathbf{V}_{\tau^{\alpha/(4-2\beta)}}(\overline{u}^{k-1}_{\tau})\right)dt\notag\\
=&\sum_{k=1}^N\frac{C_\alpha}{\tau^\alpha}\int_{(k-1)\tau}^{k\tau}\int_{\R^d}\left(u^k_{\tau}-\overline{u}^{k-1}_{\tau}\right)\psi(t)\varphi(x)dxdt\notag\\
&-\sum_{k=1}^N\frac{C_\alpha}{\tau^\alpha} \tau^{\frac{\alpha}{2(2-\beta)}} \left(\mathbf{U}(\overline{u}^{k-1}_{\tau})-\mathbf{U}(u^k_{\tau})\right)\int_{(k-1)\tau}^{k\tau}\psi(t)dt.
\end{align}
By the definition of $\overline{u}^{k-1}_{\tau}$ and using Lemma \ref{L-approximate fractional 1}, one gets that \begin{align}\label{F-main result 1}
&\sum_{k=1}^N\frac{C_\alpha}{\tau^\alpha}\int_{(k-1)\tau}^{k\tau}\int_{\R^d}\left(u^k_{\tau}-\overline{u}^{k-1}_{\tau}\right)\psi(t)\varphi(x)dxdt\notag\\
=& \sum_{k=1}^N\frac{C_\alpha}{\tau^\alpha}\int_{(k-1)\tau}^{k\tau}\int_{\R^d}\left(u^k_{\tau}-\sum_{i=0}^{k-1}(-b^{(k)}_{k-i})u^i_{\tau}\right)\psi(t)\varphi(x)dxdt\notag\\
=&  \int_{\mathbb{R}^d}\varphi(x)\sum_{k=1}^N\int_{(k-1)\tau}^{k\tau}\hat{u}_{\tau}(t)\left(C_\alpha\tau^{-\alpha}\sum_{i=k}^N b_{i-k}^{N-k}\psi(t+(i-k)\tau)\right)dtdx\notag \\
&+ \int_{\mathbb{R}^d}u(0)\varphi(x) C_\alpha\tau^{-\alpha}\sum_{k=1}^N b^{(k)}_k\int_{(k-1)\tau}^{k\tau}\psi(t)dtdx.
\end{align}
As (\ref{F-convergence of u})  one has $\mathbf{U}(\overline{u}^{k-1}_{\tau})-\mathbf{U}(u^k_{\tau})\geq 0$ for every $k\in\N$. Moreover, since (\ref{F-bounded of U 2}) and notice that $\psi\in C^\infty([0,T],\R_+)$, there exists $C_3(\psi)>0$ such that
\begin{align*}
\sum_{k=1}^N\frac{C_\alpha}{\tau^\alpha} \tau^{\frac{\alpha}{2(2-\beta)}}\left(\mathbf{U}(\overline{u}^{k-1}_{\tau})-\mathbf{U}(u^k_{\tau})\right)\int_{(k-1)\tau}^{k\tau}\psi(t)dt &\leq C_3(\psi)\tau \frac{C_\alpha}{\tau^\alpha} \tau^{\frac{\alpha}{2(2-\beta)}}\sum_{k=1}^N\left(\mathbf{U}(\overline{u}^{k-1}_{\tau})-\mathbf{U}(u^k_{\tau})\right)\\
& \leq C_3(\psi)C_\alpha \tau^{1-\alpha}\tau^{\frac{\alpha}{2(2-\beta)}}T^{1-\alpha}\mathbf{U}(u_0).
\end{align*}
Therefore, \begin{align}\label{F-main result 2}
\lim_{\tau\to 0}\sum_{k=1}^N\frac{C_\alpha}{\tau^\alpha} \tau^{\frac{\alpha}{2(2-\beta)}}\left(\mathbf{U}(\overline{u}^{k-1}_{\tau})-\mathbf{U}(u^k_{\tau})\right)\int_{(k-1)\tau}^{k\tau}\psi(t)dt=0.
\end{align}
By using \eqref{LHS-main result}, \eqref{F-main result 1}, \eqref{F-main result 2}, Lemma \ref{L-approximate fractional} and observe that $\hat{u}_{\tau}$ converges weakly to $u$ in $L^2((0,T),\dot{H}^{1-s}(\R^d))$ as $\tau\to 0$, we have \begin{align}\label{F-main result 3}
\lim_{\tau\to 0} \text{LHS}\eqref{F-main result}=\int_0^T{}_tD^\alpha_T\psi(t)\langle u(t),\varphi\rangle dt-\frac{\langle u(0),\varphi\rangle}{\Gamma(1-\alpha)}\int_0^Tt^{-\alpha}\psi(t)dt.
\end{align}
Since \eqref{F-main result 3}, in both of side of \eqref{F-main result} taking $\tau\to 0$ we obtain that
\begin{align}\label{F-main result 4}
\int_0^T{}_tD^\alpha_T\psi(t)\langle u(t),\varphi\rangle dt-\frac{\langle u(0),\varphi\rangle}{\Gamma(1-\alpha)}\int_0^Tt^{-\alpha}\psi(t)dt \leq \int_0^T \psi(t)\langle\text{div}\left( u^{\beta} \nabla(-\Delta)^{-s}u(t)\right),\varphi\rangle dt,
\end{align}
for every $\psi\in C_c^\infty([0,T],\R_+),\varphi\in C_c^\infty (\R^d)$. Similarly, by replacing $\psi$ by $-\psi$, we get the inverse inequality of \eqref{F-main result 4}. This implies that 
\begin{align*}
\int_0^T{}_tD^\alpha_T\psi(t)\langle u(t),\varphi\rangle dt =\int_0^T \psi(t)\langle\text{div}\left( u^{\beta} \nabla(-\Delta)^{-s}u(t)\right),\varphi\rangle dt+\frac{\langle u(0),\varphi\rangle}{\Gamma(1-\alpha)}\int_0^Tt^{-\alpha}\psi(t)dt.
\end{align*}	
Hence, we get the result.\\\\
$3.$ For $g(z)=z^p$ then for every $z\geq 0$ and and $\tau\leq 1$ one has \begin{align*}
\mathcal{G}(z)&=\sqrt{p(p-1)}\int_0^z\left(t+\tau^{\frac{\alpha}{4(2-\beta)}}\right)^{\beta}t^{(p-2)/2}dt\\
&\geq \sqrt{p(p-1)}\int_0^zt^{(\beta+ p-2)/2}dt\\
& = \frac{2\sqrt{p(p-1)}}{\beta + p}z^{(\beta+p)/2}.
\end{align*}
Hence, from Lemma \ref{L-estimate solution}, we have \begin{align*}
\frac{4p(p-1)\tau^\alpha}{(\beta + p)^2C_\alpha S^2_{d,1-s}}\| u^k_\tau\|_{L^{\frac{(\beta +p)d}{d-2(1-s)}}(\R^d)}^{\beta+p}\leq \|\overline{u}^{k-1}_\tau\|_{L^p(\R^d)}^p-\|u^{k}_\tau\|_{L^p(\R^d)}^p.
\end{align*}
Now, for $1\leq q<p$ we set \begin{align*}
\eta_1: = \frac{\frac{\beta +p}{q}-1+\frac{2(1-s)}{d}}{\frac{1}{q}-\frac{1}{p}},\; \eta_2 := \frac{\frac{\beta }{p}+\frac{2(1-s)}{d}}{\frac{1}{q}-\frac{1}{p}}\text{ and } \eta_3:=\beta+p.
\end{align*} 
Then we have \begin{align*}
\eta_1=\eta_2+\eta_3 \text{ and }\frac{\eta_1}{p}=\frac{\eta_2}{q}+\eta_3 \frac{d-2(1-s)}{(\beta +p)d}.
\end{align*}
Therefore, by interpolation inequality one gets that \begin{align*}
\| u^k_\tau\|_{L^{\frac{(\beta +p)d}{d-2(1-s)}}(\R^d)}^{\eta_3}\|u^k_\tau\|^{\eta_2}_{L^q(\R^d)}\geq \|u^k_\tau\|^{\eta_1}_{L^p(\R^d)}.
\end{align*}
This yields,
\begin{align}\label{F-deference u^(k-1) and u^k}
\frac{4p(p-1)\tau^\alpha}{(\beta+p)^2C_\alpha S^2_{d,1-s}}\|u^k_\tau\|^{\eta_1}_{L^p(\R^d)}\|u^k_\tau\|^{-\eta_2}_{L^q(\R^d)}\leq \|\overline{u}^{k-1}_\tau\|_{L^p(\R^d)}^p-\|u^{k}_\tau\|_{L^p(\R^d)}^p.
\end{align}
Next, for every $k\in\mathbb{N}$, we will check that \begin{align}\label{F-bounded by u_0 1}
\|u^k_\tau\|_{L^{p}(\R^d)}\leq \|u_0\|_{L^{p}(\R^d)}.
\end{align}
It is clear that \eqref{F-bounded by u_0 1} is true for $k=0$. Assume that \eqref{F-bounded by u_0 1} is true for $k$ ($k\geq 0$). From \eqref{F-deference u^(k-1) and u^k}, Lemma \ref{L-sum of b_i} and observe that $g(z)=z^p$ is convex one gets
\begin{align*}
\|u^{k+1}_\tau\|_{L^{p}(\R^d)}^p &\leq \|\overline{u}^{k}_\tau\|_{L^p(\R^d)}^p\\
& = \|\sum_{i=1}^k(-b_{k+1-i}^{(k+1)})u^i_\tau\|_{L^p(\R^d)}^p\\
& \leq \sum_{i=1}^k(-b_{k+1-i}^{(k+1)})\|u^i_\tau\|_{L^p(\R^d)}^p\\
& \leq \sum_{i=1}^k(-b_{k+1-i}^{(k+1)})\|u_0\|_{L^p(\R^d)}^p\\
& = \|u_0\|_{L^p(\R^d)}^p.
\end{align*}
Hence, by induction we get \eqref{F-bounded by u_0 1}. Similarly, by \eqref{F-bounded by u_0 1} and definition of $\overline{u}^{k-1}_\tau$, we also have 
\begin{align} \label{F-bounded by u_0 2}
\|\overline{u}^{k-1}_\tau\|_{L^{p}(\R^d)}\leq \|u_0\|_{L^{p}(\R^d)}.
\end{align}
Since \eqref{F-deference u^(k-1) and u^k}, \eqref{F-bounded by u_0 1} and \eqref{F-bounded by u_0 2} we obtain that \begin{align*}
\frac{4p(p-1)\tau^\alpha}{(\beta+p)^2C_\alpha S^2_{d,1-s}}\|u^k_\tau\|^{\eta_1}_{L^p(\R^d)}\|u_0\|^{-\eta_2}_{L^q(\R^d)}\leq \|u_0\|_{L^p(\R^d)}^p.
\end{align*}
This implies that \begin{align}\label{F-bounded by u_0 3}
 \|u^k_\tau\|_{L^p(\R^d)}\leq \bigg(\frac{(\beta + p)^2}{4p(p-1)}C_\alpha S^2_{d,1-s}\tau^{-\alpha}\|u_0\|^{\eta_2}_{L^q(\R^d)}\|u_0\|_{L^p(\R^d)}^{p}\bigg)^{\frac{1}{\eta_1}}.
\end{align}
Thus, the result follows  \eqref{F-bounded by u_0 1} and \eqref{F-bounded by u_0 3}. 
\end{proof}\\\\
\begin{proof}[Proof of Theorem \ref{main result2}]$1.$ This follows from Lemma \ref{L-convergence of u} .\\\\
$2.$ Similar as above, we also fix $T>0$ and consider $\tau=\frac{T}{N}$ for $N\in \N$. Let any $\psi\in C^\infty([0,T],\R_+)$ with $\psi(T)=0$. As $\mathbf{m}(z)=(z+1)^{\tau^{1-\alpha/4}}$ we have that 
\begin{align*}
-\lambda_\delta =\tau^{1-\alpha/4}\|D^2\varphi\|_{L^\infty}\sup_{z>0}(z+1)^{\tau^{1-\alpha/4}-1}+\frac{\tau^{1-\alpha/4}\vert \tau^{1-\alpha/4}-1 \vert}{2\delta}\|\nabla \varphi\|^2_{L^\infty}\sup_{z>0}(z+1)^{2(\tau^{1-\alpha/4}-1)}.
\end{align*}
Therefore, there exists $C_4(\varphi)>0$ such that for $\tau>0$ is small enough $$-\lambda_\delta\leq C_4(\varphi)\tau^{1-\alpha/4}(1+1/\delta).$$
By the same arguments in the proof of Theorem \ref{main result} and Lemma \ref{L-estimate of solution}, there exists $C_5(\psi,\varphi), C_6(\psi)>0$ such that \begin{align}\label{F-main result-beta=0}
&\sum_{k=1}^N\frac{C_\alpha}{\tau^\alpha}\int_{(k-1)\tau}^{k\tau}\psi(t)\left(\mathbf{V}_{\tau^{\alpha/4}}(u^k_{\tau})-\mathbf{V}_{\tau^{\alpha/4}}(\overline{u}^{k-1}_{\tau})\right)dt\notag\\
\leq & \tau^{1-\alpha/4}C_5(\psi,\varphi)-\int^{T}_0\psi(t)\langle (-\Delta)^{1-s}\hat{u}_{\tau}(t),\varphi\rangle dt+\frac{1}{2}C_6(\varphi)K_2(\psi)T^{1-\alpha}(\tau^{1+3\alpha/4}+\tau^{1+\alpha/2})\|u_0\|^2_{\dot{H}^{-s}}.
\end{align}
Moreover, 
\begin{align}\label{LHS-main result-beta=0}
\text{LHS}\eqref{F-main result-beta=0}=&\sum_{k=1}^N\frac{C_\alpha}{\tau^\alpha}\int_{(k-1)\tau}^{k\tau}\int_{\R^d}\left(u^k_{\tau}-\overline{u}^{k-1}_{\tau}\right)\psi(t)\varphi(x)dxdt\notag\\
&-\sum_{k=1}^NC_\alpha\tau^{-3\alpha/4}\left(\mathbf{U}(\overline{u}^{k-1}_{\tau})-\mathbf{U}(u^k_{\tau})\right)\int_{(k-1)\tau}^{k\tau}\psi(t)dt.
\end{align}
Since $\psi\in C^\infty([0,T],\R_+)$, there exists $C_7(\psi)>0$ such that
\begin{align*}
\sum_{k=1}^NC_\alpha\tau^{-3\alpha/4}\left(\mathbf{U}(\overline{u}^{k-1}_{\tau})-\mathbf{U}(u^k_{\tau})\right)\int_{(k-1)\tau}^{k\tau}\psi(t)dt &\leq C_7(\psi)\tau C_\alpha\tau^{-3\alpha/4}\sum_{k=1}^N\left(\mathbf{U}(\overline{u}^{k-1}_{\tau})-\mathbf{U}(u^k_{\tau})\right)\\
& \leq C_7(\psi)C_\alpha \tau^{\alpha/4}T^{1-\alpha}\mathbf{U}(u_0).
\end{align*}
This implies that \begin{align}\label{F-main result 2-beta=0}
\lim_{\tau \to 0}\sum_{k=1}^NC_\alpha\tau^{-3\alpha/4}\left(\mathbf{U}(\overline{u}^{k-1}_{\tau})-\mathbf{U}(u^k_{\tau})\right)\int_{(k-1)\tau}^{k\tau}\psi(t)dt=0.
\end{align}
From \eqref{F-main result 1}, \eqref{LHS-main result-beta=0}, \eqref{F-main result 2-beta=0}, Lemma \ref{L-approximate fractional} and $\hat{u}_{\tau}$ converges weakly to $u$ in $L^2((0,T),\dot{H}^{1-s}(\R^d))$ as $\tau\to 0$, we have 
\begin{align*}
\lim_{\tau\to 0} \text{LHS}\eqref{F-main result-beta=0}=\int_0^T{}_tD^\alpha_T\psi(t)\langle u(t),\varphi\rangle dt-\frac{\langle u(0),\varphi\rangle}{\Gamma(1-\alpha)}\int_0^Tt^{-\alpha}\psi(t)dt.
\end{align*}
Hence, from this equality, taking $\tau\to 0$ in both of side of \eqref{F-main result-beta=0} one gets that
\begin{align*}
\int_0^T{}_tD^\alpha_T\psi(t)\langle u(t),\varphi\rangle dt-\frac{\langle u(0),\varphi\rangle}{\Gamma(1-\alpha)}\int_0^Tt^{-\alpha}\psi(t)dt \leq -\int_0^T \psi(t)\langle (-\Delta)^{1-s}u(t),\varphi\rangle dt,
\end{align*}
for every $\psi\in C_c^\infty([0,T],\R_+),\varphi\in C_c^\infty (\R^d)$. \\
Therefore, by replacing $\psi$ by $-\psi$, we get the result.\\\\
$3.$ For $g(z)=z^p$ and $m(z)=(z+1)^{\tau-\alpha/4}$, we have that
\begin{align*}
\mathcal{G}(z)&=\sqrt{p(p-1)}\int_0^z(t+1)^{\tau^{1-\alpha/4}/2}t^{(p-2)/2}dt\\
&\geq \sqrt{p(p-1)}\int_0^zt^{(\tau^{1-\alpha/4}+ p-2)/2}dt\\
& = \frac{2\sqrt{p(p-1)}}{\tau^{1-\alpha/4}+p}z^{(\tau^{1-\alpha/4}+p)/2}\\
&\geq \frac{2\sqrt{p(p-1)}}{p+1}z^{(\tau^{1-\alpha/4}+p)/2}.
\end{align*}
Furthermore, if we set $\theta_3:=\tau^{1-\alpha/4}+p$ then
\begin{align*}
\theta_1=\theta_2+\theta_3 \text{ and }\frac{\theta_1}{p}=\frac{\theta_2}{q}+\theta_3 \frac{d-2(1-s)}{(\tau^{1-\alpha/4} +p)d}.
\end{align*}
Therefore, by the same method as in the proof of Theorem \ref{main result} (3), we get the result.
\end{proof}



\begin{thebibliography}{99}
\bibitem{ACV} M. Allen, L. Caffarelli, A. Vasseur, {\it Porous medium flow with both a fractional potential pressure and fractional time derivative}, Chin. Ann. Math. Ser. B,  {\bf 38} (2017), 45--82.

\bibitem{AmGiSa} L. Ambrosio, N. Gigli, G. Savare, {\it Gradient Flows in Metric Spaces and in the Space of Probability Measures,} second ed., Lectures in Math. ETH Zuurich, Birkhauser Verlag, Basel, 2008.

\bibitem{AMS} L. Ambrosio, E. Mainini, S. Serfaty, {\it Gradient flow of the Chapman-Rubinstein-Schatzman model for signed vortices}, Annales IHP, Analyse non lin\'{e}aire, {\bf 28} (2011), no. 2,  217--246.

\bibitem{AS} L. Ambrosio, S. Serfaty,  {\it A gradient flow approach to an evolution problem arising in superconductivity}, Comm. Pure Appl. Math., {\bf 61} (2008), no. 11, 1495--1539.

\bibitem{A} D.G. Aronson, {\it The Porous Medium Equation}, Nonlinear Diffusion Problems, Lecture Notes in Math. Vol. 1224 (Eds. Fasano A. and Primicerio M.) Springer, New York, 12--46, 1986.

\bibitem{BB} J. D. Benamou, J. Brenier, {\it A computational fluid mechanics solution to the Monge-Kantorovich mass transfer problem}, Numerische Mathematik, {\bf 84.3} (2000), 375--393.

\bibitem{BCD} H. Bahouri, J-Y. Chemin, R. Danchin, {\it Fourier analysis and nonlinear partial differential equations}, Grundlehren der Mathematischen Wissenschaften [Fundamental Principles of Mathematical Sciences], 343. Springer, Heidelberg, 2011.
	
\bibitem{CV} L. A. Caffarelli, J. L. V\'{a}zquez, {\it Nonlinear porous medium flow with fractional potential pressure}, Arch. Rational Mech. Anal., {\bf 202} (2011), 573--565.

\bibitem{Caputo} M. Caputo, {\it Diffusion of fluids in porous media with memory}, Geothermics, {\bf 28} (1999), 113--130.

\bibitem{CLSS} J. Carrillo, S. Lisini, G. Savar\'{e}, D. Slep\v{c}ev, {\it Nonlinear mobility continuity equations and generalized displacement convexity}, J. Funct. Anal., {\bf 258} (2010), 1273--1309.
	
\bibitem{CN} N-P. Chung, Q-H. Nguyen, {\it Gradient flows of modified Wasserstein distances and porous medium equations with nonlocal pressure},  Acta Mathematica Vietnamica, {\bf 48} (2023), 209--235.

\bibitem{DS} S. Daneri, G. Savar\'{e}, {\it Eulerian calculus for the displacement convexity in the Wasserstein distance}, SIAM J. Math. Anal., {\bf 40} (2008), no. 3, 1104--1122.

\bibitem{DN} N. A. Dao, V. T. A. Nguyen, {\it Decay estimates for time-fractional porous medium flow with nonlocal pressure}, arXiv:2304.12580.

\bibitem{DNA} J.-D. Djida, J. J. Nieto, I. Area, {\it Nonlocal time-porous medium equation: Weak solutions and finite speed of propagation}, Discrete and Continuous Dynamical Systems-B, {\bf 24} (2019), 4031--4053.

\bibitem{DNS} J. Dolbeault, B. Nazaret, G. Savar\'{e}, {\it A new class of transport distances between measures}, Calc. Var. Partial Differential Equations., {\bf 34} (2009), 193--231.

\bibitem{DJ} M. H. Duong and B. Jin, {\it Wasserstein Gradient Flow Formulation of the Time-Fractional Fokker-Planck Equation}, Communications in Mathematical Sciences, {\bf 18} (2020), no. 7, 1949--1975.

\bibitem{DPZ} M. H. Duong, M. A. Peletier, and J. Zimmer, {\it Conservative-dissipative approximation schemes for a generalized Kramers equation}, Math. Methods Appl. Sci.,  {\bf 37} (2014), no. 16, 2517--2540.

\bibitem{Erbar} M. Erbar, {\it Gradient flows of the entropy for jump processes}, Ann. Inst. Henri Poincaré Probab. Stat., {\bf 50} (2014), no. 3, 920--945.

\bibitem{H} C. Huang, A variational principle for the Kramers equation with unbounded external forces. Journal of Mathematical Analysis and Applications, {\bf250}(1) (2000), 333--367. 

\bibitem{JKO} R. Jordan, D. Kinderlehrer, F. Otto, {\it The variational formulation of the Fokker–Planck equation}, SIAM J. Math. Anal., {\bf 29} (1998), 1--17.

\bibitem{KSVZ} J. Kemppainen, J. Siljander, V. Vergara, R. Zacher, {\it Decay estimates for time-fractional and
other non-local in time subdiffusion equations in $\R^d$}, Math. Ann., {\bf 366} (2016), 941--979.

\bibitem{KSZ}J. Kemppainen, J. Siljander, R. Zacher, {\it Representation of solutions and large time behavior for fully
nonlocal diffusion equations}, J. Differential Equations, {\bf 263} (2017), 149--201.

\bibitem{KST} A. Kilbas, H. M. Srivastava, J. J. Trujillo, {\it Theory and Applications of Fractional Differential Equations}, Elsevier Science B. V., Amsterdam, (2006), 2.1



\bibitem{KRY} A. Kubica, K. Ryszewska, M. Yamamoto, {\it Theory of Time-fractional Differential Equations An Introduction}, Springer Japan, Tokyo, 2020.



\bibitem{LLW} L. Li, J-G. Liu, L. Wang, {\it Cauchy problems for Keller–Segel type time-space fractional diffusion equation}, J. Differential Equations, {\bf 265} (2018), 1044--1096.

\bibitem{LX} Y. Lin and C. Xu. {\it Finite difference/spectral approximations for the time-fractional diffusion equation. J. Comput. Phys.}, {\bf 225} (2) (2007), 1533--1552.

\bibitem{LMS} S. Lisini, E. Mainini, A. Segatti, {\it A gradient flow approach to the porous medium equation with fractional pressure},  Arch. Ration. Mech. Anal., {\bf 227} (2018), 567--606 .

\bibitem{LMS12} S. Lisini, D. Matthes, G. Savar\'{e}, {\it Cahn-Hilliard and thin film equations with nonlinear mobility as gradient flows in weighted - Wasserstein metric}, J. Differential Equations., {\bf 253} (2012), 814--850.

\bibitem{MP} F. Mainardi, G. Pagnini, {\it The Wright functions of the time-fractional diffusion equation}, Applied Mathematics and Computation, {\bf 141} (2003), 51--62.
	
\bibitem{MMS}D. Matthes, R. McCann, G. Savar\'{e}, {\it A family of nonlinear fourth order equations of gradient flow type}. Commun. Partial Differ. Equ., {\bf 34} (2009), 1352--1397.	

\bibitem{NV} Q.-H. Nguyen, J. L. V'{a}zquez, {\it Porous medium equation with nonlocal pressure in a bounded domain}, Comm. Partial Diff. Equ., {\bf 43} (2017), 1502--1539.

\bibitem{Otto} F. Otto, {\it The geometry of dissipative evolution equations: the porous medium equation}, Comm. Partial Differential Equations, {\bf 26} (2001), 101--174.

\bibitem{OttoWest} F. Otto, M. Westdickenberg, {\it  Eulerian calculus for the contraction in the Wasserstein distance}, SIAM J. Math Anal., {\bf 37} (2005), 1227--1255.

\bibitem{PRST} M. A. Peletier, R. Rossi, G. Savaré, and O. Tse, Jump processes as generalized gradient flows. Calc. Var. Partial Differential Equations, {\bf 61} (2022), no. 1, Paper No. 33, 85 pp.

\bibitem{STV-JDE} D. Stan, F. del Teso, J. L. V'{a}zquez, {\it Finite and infinite speed of propagation for porous medium equations with fractional pressure}, J. Diff. Equ., {\bf 260} (2) (2016), 1154--1199 .

\bibitem{STV-ARMA} D. Stan, F. del Teso, J. L. V'{a}zquez, {\it Existence of weak solutions for a general porous medium equation with nonlocal pressure}, Arch Rational Mech. Anal., {\bf 233} (2019), 451--496.

\bibitem{Vq} J. L. V\'{a}zquez, {\it The Porous Medium Equation. Mathematical theory}, Oxford Mathe- matical Monographs. The Clarendon Press/Oxford University Press, Oxford, 2007.



\bibitem{V} C. Villani, {\it Topics in optimal transportation}, Graduate Studies in Mathematics, vol. 58, Providence, RI, Amer. Math. Soc., (2003).

\end{thebibliography}
\end{document}